\documentclass[letterpaper, 10 pt, conference]{ieeeconf}  
\IEEEoverridecommandlockouts    
\usepackage{comment}
\usepackage{cite}
\newcommand{\eat}[1]{}
\usepackage{booktabs}
\usepackage{color}
\usepackage [autostyle, english = american]{csquotes}
\MakeOuterQuote{"}
\usepackage{xcolor}

\usepackage{multicol}

\usepackage{mathtools}
    
\usepackage{amsmath}
\usepackage{amstext}
\usepackage{amssymb}
\usepackage{amsfonts}
\usepackage{float}

\usepackage{amsthm}  
\newtheorem{theorem}{Theorem}
\newtheorem{lemma}{Lemma}
\newtheorem{definition}{Definition}
\newtheorem{corollary}{Corollary}

\newtheorem{remark}{Remark}

\newtheorem{proposition}{Proposition}

\newtheorem{assumption}{Assumption}
\newtheorem{problem}{Problem}
\usepackage{subfig}
\usepackage{caption}
\usepackage{todonotes}

\makeatletter

\newcommand{\Rmnum}[1]{\expandafter\@slowromancap\romannumeral #1@}
\usepackage{savesym}

\usepackage{algorithm}
\usepackage{algorithmicx} 
\usepackage{algpseudocode} %
\savesymbol{AND}
\usepackage[group-separator={,},group-minimum-digits={3}]{siunitx}

\usepackage{graphicx} 
\usepackage{epsfig} 

\usepackage{times} 
\usepackage{amsmath} 
\usepackage{amssymb}  
\usepackage{comment}
\makeatletter
\let\NAT@parse\undefined
\makeatother
\usepackage{hyperref}
\hypersetup{
   colorlinks=true,
    linkcolor= blue,
    allcolors=blue,
    citecolor = blue,
    filecolor=black,      
    urlcolor=blue,
    }
\usepackage{mathrsfs}
\DeclareMathOperator*{\E}{\mathbb{E}}
\DeclareMathOperator\erf{erf}

\usepackage{blindtext}

\title{\LARGE \bf
Robust Learning-Based Trajectory Planning for Emerging Mobility Systems }
\author{Behdad Chalaki, \emph{IEEE Student Member}, Andreas A. Malikopoulos, \emph{IEEE Senior Member}%
\thanks{This research was supported in part by ARPAE's NEXTCAR program under the award number DE- AR0000796 and by the Delaware Energy Institute (DEI). This support is gratefully acknowledged.}%
\thanks{The authors are with the Department of Mechanical Engineering, University of Delaware, Newark, DE 19716 USA (emails: \texttt{\{bchalaki;andreas\}@udel.edu}).}}

\begin{document}

\maketitle
\thispagestyle{empty}
\pagestyle{empty}

\begin{abstract}
In this paper, we extend a framework that we developed earlier for coordination of connected and automated vehicles (CAVs) at a signal-free intersection to incorporate uncertainty. Using the possibly noisy observations of actual time trajectories and leveraging Gaussian process regression, we learn the bounded confidence intervals for deviations from the nominal trajectories of CAVs online.
Incorporating these confidence intervals, we reformulate the trajectory planning as a robust coordination problem, the solution of which guarantees that constraints in the system are satisfied in the presence of bounded deviations from the nominal trajectories. We demonstrate the effectiveness of our extended framework through a numerical simulation.

\end{abstract}

\section{Introduction}

\PARstart{T}{he} growing population in urban areas has led to traffic congestion, increasing delays, and environmental concerns \cite{Schrank2019}. However, the introduction of communication technologies and computational capabilities into connected and automated vehicles (CAVs) has the potential to address these challenges. Through these advancements, we are transitioning to an \textit{emerging mobility system}, in which CAVs can make better decisions leading to reductions of energy consumption, travel delays, and improvements to passengers safety \cite{Wadud2016}. 

After the seminal work of Athans \cite{Athans1969} on safely coordinating vehicles at merging roadways, several research efforts have explored the benefits of coordinating CAVs in traffic scenarios using a bi-level framework which consists of (1) an \emph{upper-level} optimization that yields, for each CAV, the optimal time to exit a predetermined control zone of the intersection; and (2) a \emph{low-level} optimization that yields for the CAV its optimal control input (acceleration/deceleration) to achieve the optimal time derived in the upper-level subject to the state, control, and safety constraints. 
There have been a variety of approaches in the literature to solve the upper-level optimization problem, including first-in-first-out (FIFO) queuing policy 
\cite{Malikopoulos2017,Rios-Torres2017b,Beaver2020DemonstrationCity}, heuristic Monte Carlo tree search methods \cite{xu2019cooperative,xu2020bi}, centralized optimization techniques \cite{guney2020scheduling,hult2018optimal}, and job-shop scheduling \cite{chalaki2020TITS, fayazi2018mixed}. Given the solution of the upper-level optimization problem, the low-level optimization for each CAV addresses a constrained optimal control problem using model predictive control (MPC) \cite{hult2018optimal,kim2014mpc,campos2014cooperative,kloock2019distributed,du2018hierarchical}, or standard optimal control techniques resulting in closed-form analytical solutions  \cite{Malikopoulos2017,chalaki2020TCST,Malikopoulos2020,zhang2019decentralized}. However, the latter approach leads to a system of non-linear equations that might be challenging, in some instances, to solve in real time. To address this problem, a different approach was recently proposed in \cite{Malikopoulos2020} consisting of a single optimization level aimed at both minimizing energy consumption and improving the traffic throughput. In this approach, each CAV computes the optimal exit time of the control zone corresponding to an unconstrained energy optimal trajectory which satisfies all the state, control, and safety constraints.

Although there have been several studies addressing the problem of coordination of CAVs, only a limited number of papers considers uncertainty. Xiao et al. \cite{xiao2020bridging} employed a control barrier function (CBF) to track the optimal control trajectory in the presence of noise process in the model. 
Zhou et al.\cite{zhou2017rolling} proposed a centralized receding horizon stochastic optimal control strategy for adaptive cruise control and cooperative adaptive cruise control of platoons of vehicles to incorporate noise in the system dynamics and measurements. 
\eat{However, their computational time \eat{of this approach}increases exponentially with respect to the prediction horizon and platoon size.}
In another effort \cite{chalaki2020TCST}, a known bounded steady-state error was considered in tracking the optimal position of the CAVs \eat{resulted from a bi-level framework }aimed at minimizing energy consumption and improving the traffic throughput. \eat{Several other efforts have considered addressing the uncertainty in the single autonomous vehicle using sampling-based techniques \cite{berntorp2019motion,vitus2013probabilistic}, stochastic MPC
\cite{carvalho2016predictive}, stochastic reachability method \cite{gao2019stochastic}, and  scenario-based MPC \cite{schildbach2015scenario}.}

In this paper, we build upon the framework introduced in \cite{Malikopoulos2020} and enhance it by reformulating the coordination of CAVs as a robust coordination problem. We employ Gaussian process (GP) regression to learn the deviation of CAVs from their nominal time trajectory and obtain confidence intervals on the unknown errors of nominal trajectories based on the noisy observations of CAVs. The obtained confidence intervals can then be used to solve the robust coordination problem using a worst-case scenario approach.  
A GP is defined as a collection of random variables, any finite number of which have a joint Gaussian distribution, and can be used to describe a distribution over an infinite-dimensional space of functions  \cite{rasmussen2003gaussian}.
GP models have been used in various control applications, including ship trajectory prediction \cite{rong2019ship}, modeling and control of buildings \cite{jain2018learning}, and safe-learning for ground and aerial vehicles \cite{ostafew2016robust,berkenkamp2018safe,fisac2018general}.

To the best of our knowledge, this is the first attempt to model uncertainty in decentralized coordination of CAVs. Therefore, we believe that this paper advances the state of the art in the following ways. First, rather than not considering uncertainty for the vehicle's trajectory planning \cite{Malikopoulos2017,Malikopoulos2020,zhang2019decentralized,guney2020scheduling,hult2018optimal,kim2014mpc,campos2014cooperative,kloock2019distributed,du2018hierarchical} or assuming a known bound \cite{chalaki2020TCST}, we employ GP regression to model uncertainty and incorporate it in our coordination framework. Second, by considering uncertainty in the vehicle's trajectory planning, we ensure safety in the presence of uncertainty without sacrificing optimality, in contrast to the methods with reactive mechanisms such as CBF \cite{xiao2020bridging} which can potentially result in an optimality gap. 

The rest of the paper is structured as follows. 
In Section \Rmnum{2}, we introduce the modeling framework, and in Section \Rmnum{3}, we present the solution approach. We provide simulation results in Section \Rmnum{4}, and concluding remarks \eat{along with a discussion for a future research direction }in Section \Rmnum{5}.

\section{Problem Formulation} 
We consider a signal-free intersection, which includes a \textit{coordinator} that stores information about the intersection's geometry and CAVs' trajectories (Fig. \ref{fig:intersection}). The coordinator does not make any decision, and it only acts as a \textit{\mbox{{database}}} for the CAVs. The intersection includes a \textit{{control zone}} inside of which the CAVs can communicate with the coordinator. We call the points inside the control zone where paths of CAVs intersect and a lateral collision may occur \textit{conflict points}. Let $\mathcal{L}\subset \mathbb{N}$ indexes the set of conflict points, $N(t)\in\mathbb{N}$ be the total number of CAVs inside the control zone at time $t\in\mathbb{R}_{\geq0}$, and $\mathcal{N}(t)=\{1,\ldots,N(t)\}$ be the queue that designates the order in which each CAV entered the control zone. We model the dynamics of each CAV $i\in\mathcal{N}(t)$ as a double integrator
\begin{align}
\begin{aligned}\label{eq:dynamics}
\dot{p}_i(t)=v_i(t),\\
\dot{v}_i(t)=u_i(t),
\end{aligned}\end{align}
where $p_{i}(t)\in\mathcal{P}_{i}$, $v_{i}(t)\in\mathcal{V}_{i}$, and
$u_{i}(t)\in\mathcal{U}_{i}$ denote position, speed, and control input at $t$, respectively. 
The sets $\mathcal{P}_{i}$,
$\mathcal{V}_{i}$, and $\mathcal{U}_{i}$, $i\in\mathcal{N}(t),$
are compact subsets of $\mathbb{R}$. 
Let $\mathbf{x}_{i}(t)=[p_{i}(t)~ v_{i}(t)]^\top$ be the state of the CAV $i$ at time $t$, $t_{i}^{0}\in\mathbb{R}_{\geq 0}$ be the time that CAV $i\in\mathcal{N}(t)$ enters the control zone, and $t_{i}^{f}>t_i^0\in\mathbb{R}_{\geq 0}$ be the time that CAV $i$ exits the control zone. 
The state of CAV $i$, upon entering the control zone, is given by $\mathbf{x}_{i}(t_i^0)=[p_{i}^0~ v_{i}^0]^\top$, where $p_i^0 = p_i(t_i^0)=0$ and $v_i^0 = v_i(t_i^0)$. Similarly, the final state is denoted by $\mathbf{x}_{i}(t_i^f)=[p_{i}^f~ v_{i}^f]^\top$, where $p_i^f = p_i(t_i^f)$ and $v_i^f= v_i(t_i^f)$. \eat{It is worth mentioning that $p_i^f$ has a known value depending on the geometry of the intersection and path of CAV, i.e., (1) $p_i^f = 2L + D$, if CAV crosses the merging zone without turn; (2) $p_i^f = 2L + \frac{\pi R_r}{2}$, if CAV makes a right turn at the merging zone; and (3) $p_i^f = 2L + \frac{\pi R_l}{2}$, if CAV makes a left turn at the merging zone, where $R_r\in\mathbb{R}_{>0}$ and $ R_l\in\mathbb{R}_{>0}$ denote radius for the right turn and left turn, respectively.}
For each CAV $i\in\mathcal{N}(t)$ the control input and speed are bounded by 
\begin{align}
    u_{i,\min}&\leq u_i(t)\leq u_{i,\max}, \label{eq:uconstraint} \\
    0< v_{\min}&\leq v_i(t)\leq v_{\max} \label{eq:vconstraint},
\end{align}
where $u_{i,\min},u_{i,\max}$ are the minimum and maximum control inputs and $v_{\min},v_{\max}$ are the minimum and maximum speed limit, respectively. 

\begin{figure}[t]
    \centering
\includegraphics[width=0.95\linewidth]{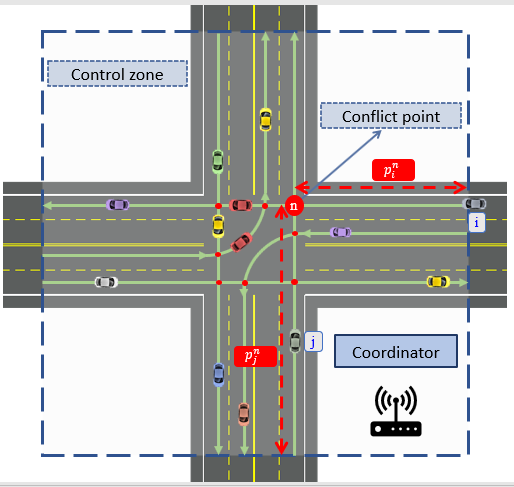}  \caption{A signal free intersection with conflict points.}
    \label{fig:intersection}
\end{figure}

To guarantee rear-end safety between CAV $i\in\mathcal{N}(t)$ and the preceding CAV $k\in\mathcal{N}(t)\setminus\{i\}$, we have 
\begin{gather}\label{eq:rearend}
 p_k(t)-p_i(t)\geq \delta_i(t)= \gamma + \varphi\cdot v_i(t),
\end{gather}
where $\delta_i(t)$ is the safe speed-dependent distance, while $\gamma$ and $\varphi\in\mathbb{R}_{>0}$ are the standstill distance and reaction time, respectively. 
\begin{definition}\label{def:timetraj}
For CAV $i\in\mathcal{N}(t)$, the inverse function, $p_i^{-1}$ is the time trajectory and denoted by $t_i\colon \mathcal{P}_{i} \rightarrow [t_i^0, t_i^f]$. The time trajectory yields the time that CAV $i$ is at position $p_i\in\mathcal{P}_i$ inside the control zone.
\end{definition}
Since $ 0 < v_{\min} \leq v_i(t)$, the position $p_i(t)$ is a strictly increasing function. Moreover, for every element in $\mathcal{P}_i$, there is at least one element in $[t_i^0, t_i^f]$, which implies that the position $p_i(t)$ is a surjective function \cite{Malikopoulos2020}, and hence the inverse $t_i\left(\cdot\right) = p_i^{-1}\left(\cdot\right)$ exists. 
However, for the cases where $v_{\min}=0$, for any $p\in\mathcal{P}_i$, we can use $t_i (p) =p_i^{-1}(p)=\{\max (\tau)~\vert~p_i(\tau)=p~,~ \tau\subset{[t_i^0, t_i^f]}\}$, which is the maximum preimage of $p$.

Let CAV $j\in\mathcal{N}(t) \setminus \{i\}$ be a CAV that is inside the control zone and has already planned its trajectory. CAV $i$ may have a lateral collision with CAV $j$ at conflict point $n\in\mathcal{L}$ (Fig. \ref{fig:intersection}). 
We denote by $p_{i}^n$ and $p_{j}^n$ the distance of the conflict point $n$ from $i$'s and $j$'s paths' entries, respectively (Fig. \ref{fig:intersection}).
To guarantee lateral collision avoidance, we impose the following time headway constraint
\begin{align} \label{eq:lateralSafety}
    &\big|t_i(p_i^n) - t_j(p_j^n)\big| \geq t_h,
\end{align}
where $t_h\in\mathbb{R}_{>0}$ is the minimum time headway between any two CAVs crossing conflict point $n$.

\eat{Next, we present deterministic decentralized optimal control problems for each CAV $i\in\mathcal{N}(t)$ over the interval $t\in[t_i^0, t_i^f]$ \cite{Malikopoulos2020}.}

\subsection{Deterministic Coordination Problem}
We start our exposition by briefly reviewing the single-level optimization framework for coordination of CAV developed in \cite{Malikopoulos2020}. 
In our framework, upon entrance, each CAV $i\in\mathcal{N}(t)$ communicates with the coordinator to access the time trajectories of CAVs which are already in the control zone. After obtaining this information, CAV $i$ solves a time minimization problem to determine the time that it must exit the control zone, $t_i^f$. The time $t_i^f$ corresponds to the unconstrained energy optimal trajectory guaranteeing that state, control, and safety constraints are satisfied. This trajectory is communicated back to the coordinator, so that the subsequent CAVs receive this information and plan their trajectories accordingly. 
Our framework implies that the CAVs do not have to come to a full stop at the intersection, thereby conserving momentum and fuel while also improving travel time. 

This approach allows us to handle the trade-off between minimizing energy consumption and travel time by having an energy-optimal trajectory combined with the minimum travel time. 
Furthermore, by enforcing the unconstrained energy optimal trajectory that guarantees the satisfaction of all the state, control, and safety constraints, we avoid inherent real-time implementation difficulties in solving a constrained optimal control problem. To solve a constrained optimal control problem, one can employ the standard methodology of Hamiltonian analysis with interior point state and/or control constraints \cite{bryson1975applied}. Namely, we first start with the unconstrained solution to the optimal control problem. If the solution violates any of state or control constraints, then the unconstrained arc is pieced together with the arc corresponding to the violated constraint at unknown time $\tau_1$, and we re-solve the problem with the two arcs pieced together. The two arcs yield a set of algebraic equations which are solved simultaneously using the boundary and interior conditions at $\tau_1$. If the resulting solution violates another constraint, then the last two arcs are pieced together with the arc corresponding to the new violated constraint, and we re-solve the problem with the three arcs pieced together at unknown times $\tau_1$ and $\tau_2$. The three arcs will yield a new set of algebraic equations that need to be solved simultaneously using the boundary and interior conditions at $\tau_1$ and $\tau_2$. The process is repeated until the solution does not violate any other constraints. Since this problem leads to a system of non-linear equations, in some cases, this iterative process can be computationally expensive.

Our exposition for the single-level optimization framework starts with the unconstrained energy optimal solution of CAV $i$ which has the following form \cite{Malikopoulos2020}  
\begin{alignat}{3}
p_{i}(t) &= \phi_{i,3} \cdot t^3 +\phi_{i,2} \cdot t^2 + \phi_{i,1} \cdot t +\phi_{i,0},&& \label{eq:nominal_p}\\
	v_{i}(t) & = 3\phi_{i,3} \cdot t^2 +2\phi_{i,2} \cdot t + \phi_{i,1}, && \label{eq:nominal_v} \\
	u_{i}(t) & = 6 \phi_{i,3} \cdot t + 2 \phi_{i,2}, && \label{eq:nominal_u}
\end{alignat}
where $\phi_{i,3}\neq 0$ and $\phi_{i,2}, \phi_{i,1}, \phi_{i,0}\in\mathbb{R}$ are the constants of integration. CAV $i$ must also satisfy the boundary conditions
\begin{align}
     p_i(t_i^0) &= 0,\quad  v_i(t_i^0)= v_i^0 , \label{eq:bci}\\
     p_i(t_i^f)&=p_i^f,\quad u_i(t_i^f)=0, \label{eq:bcf}
\end{align}
where $u_i(t_i^f)=0$ because the speed at the exit of the control zone is not specified \cite{bryson1975applied}.

Using the Cardano's method \cite{cardano2007rules}, the time trajectory $t_i(p_i)$ in Definition \ref{def:timetraj} is given by
\begin{multline}\label{eq:nominal_time}
	t_i(p_i)= \\
 \sqrt[3]{ - \frac{1}{2} \left(\omega_{i,1} + \omega_{i,2}~ p_i \right) + \sqrt{\frac{1}{4} \left(\omega_{i,1} + \omega_{i,2}~ p_i \right) ^ 2 + \frac{1}{27}\omega_{i,0} ^ 3}} +\\
\sqrt[3]{ - \frac{1}{2} \left(\omega_{i,1} + \omega_{i,2}~ p_i \right) - \sqrt{\frac{1}{4} \left(\omega_{i,1} + \omega_{i,2}~ p_i \right) ^ 2 + \frac{1}{27}\omega_{i,0} ^ 3}}\\
+ \omega_{i,3},\quad \quad p_i \in \mathcal{P}_i,
	\end{multline}
	\eat{
	\begin{align}
	\omega_{i,0} &= \frac{\phi_{i,1}}{\phi_{i,3}} - \frac{1}{3}\left(\frac{\phi_{i,2}}{\phi_{i,3}}\right)^2, \label{eqn:omega0}\\
	\omega_{i,1} &= \frac{1}{27}\left[2\left(\frac{\phi_{i,2}}{\phi_{i,3}}\right)^3 - \frac{9 \phi_{i,2} \cdot \phi_{i,1}}{(\phi_{i,3}) ^ 2} \right] + \frac{\phi_{i,0}}{\phi_{i,3}} \label{eqn:omega1}, \\
	\omega_{i,2} &= - \frac{1}{\phi_{i,3}}, \quad\quad
	\omega_{i,3} = - \frac{\phi_{i,2}}{3 \phi_{i,3}}.\label{eqn:omeg2,a3}
\end{align}}
where $\omega_{i,3}, \omega_{i,2}, \omega_{i,1}$, and $\omega_{i,0}\in\mathbb{R}$ such that $\frac{1}{4}(\omega_{i,1} + \omega_{i,2}~ p_i) ^ 2 + \frac{1}{27}\omega_{i,0} ^ 3 > 0$, and they are all defined in terms of $\phi_{i,3}$, $ \phi_{i,2}$, $ \phi_{i,1}$, $ \phi_{i,0}\in\Phi_i$, $\Phi_i\subset\mathbb{R}$, with $\phi_{i,3} \neq 0$.  The algebraic derivation of \eqref{eq:nominal_time} is tedious but standard \cite{Malikopoulos2020}, and thus omitted. 

Next, we formally define the single-level optimization framework aimed at minimizing the exit time from the control zone. 
\begin{problem}\label{prb:mintfProblem}
Each CAV $i\in\mathcal{N}(t)$ solves the following optimization problem at $t_i^0$, upon entering the control zone 
\end{problem}

\begin{align}\label{eq:tif}
    &\min_{t_i^f\in \mathcal{T}_i(t_i^0)} t_i^f \\
    \text{subject to: }& \notag\\
    & \eqref{eq:rearend}, \eqref{eq:lateralSafety},\eqref{eq:nominal_p},\eqref{eq:nominal_v}, \notag \\
    &\eqref{eq:nominal_u}, \eqref{eq:bci}, \eqref{eq:bcf},\eqref{eq:nominal_time},  \notag
\end{align}
\textit{where the compact set $\mathcal{T}_i(t_i^0)$ is the set of feasible solution of CAV $i\in\mathcal{N}(t)$ for the exit time computed at $t_i^0$ using the speed and control input constraints \eqref{eq:uconstraint}-\eqref{eq:vconstraint}, initial condition \eqref{eq:bci}, and final condition \eqref{eq:bcf}. The derivation of this compact set is discussed in \cite{chalaki2020experimental}. }
\begin{remark}
The solution of Problem \ref{prb:mintfProblem} yields a $t_i^f$ which guarantees that none of the constraints in Problem \ref{prb:mintfProblem} becomes active, and thus CAV $i$ follows the unconstrained energy-optimal solution \eqref{eq:nominal_p}-\eqref{eq:nominal_u}.
\end{remark}

\subsection{Uncertainty in the Coordination Problem}
In an earlier work \cite{chalaki2020experimental}, we showed that there is a discrepancy between the actual and planned trajectories due to the presence of uncertainty originated from error in low-level tracking, noisy measurements, etc.
In this paper, to accommodate this uncertainty, we reformulate Problem \ref{prb:mintfProblem} as a robust coordination problem, the solution of which guarantees that constraints in the system are satisfied in the presence of bounded deviations from the nominal trajectories.
\eat{This predetermined probability can be used to tune the level of the conservativeness of the resulting solution to the underlying problem.We proceed with introducing the following definitions to formally state the problem.}
\begin{assumption}
The deviation from the deterministic nominal time trajectory of a real-physical CAV $i\in\mathcal{N}(t)$ is given by an unknown function $e_i:\mathcal{P}_i \rightarrow\mathcal{E}_i$, where $\mathcal{E}_i$ is an unknown bounded subset of $\mathbb{R}$. 
\end{assumption}
We consider that $e_i(p_i)$ can be approximated by a Gaussian process defined on a probability space $(\mathcal{P}_i,\mathscr{P}_i,\mathbb{P})$, where $\mathscr{P}_i$ is the associated $\sigma$-algebra and $\mathbb{P}$ is a probability measure on $(\mathcal{P}_i,\mathscr{P}_i)$. This is a reasonable approach, since GP regression has been used successfully to approximate functions in many applications \cite{rasmussen2003gaussian}.

\begin{definition}
The actual time trajectory for CAV $i\in\mathcal{N}(t)$ is a random process defined on $(\mathcal{P}_i,\mathscr{P}_i,\mathbb{P})$, denoted by $\hat{t}_{i}:\mathcal{P}_i \rightarrow\mathbb{R}$, and given by
\begin{equation}
    \hat{t}_{i}(p_i) = t_i(p_i)+ e_i(p_i),\label{eq:actual_t}
\end{equation}
where $t_i(p_i)$ is the nominal trajectory which is the solution of Problem \ref{prb:mintfProblem}. \end{definition}

From Definition \ref{def:timetraj}, the time trajectory is the inverse function of the position trajectory. Having a deviation in the time trajectory also makes the deviation in the position trajectory inevitable.  

\begin{definition}\label{def:actPosTraj}
The actual position trajectory of CAV $i\in\mathcal{N}(t)$ is a random process denoted by $\hat{p}_i$ defined on a probability space $(\Omega_i,\mathscr{F}_i,\mathbb{P})$, $\Omega_i\in\mathbb{R}$, and given by
\begin{equation}
    \hat{p}_{i}(t)= p_i(t)+ f_i(t),  \label{eq:actual_p}\\
\end{equation}
where $f_i(t)$ is the unknown deviation from the nominal position trajectory, and it is defined on $(\Omega_i,\mathscr{F}_i,\mathbb{P})$.
\end{definition}

\begin{lemma}\label{lem:actualPositionTraj}
The deviation in the position trajectory of CAV $i\in\mathcal{N}(t)$, $f_i(t)$, can be derived from the deviation $e_i(t)$ of its time trajectory.

\end{lemma}
\begin{proof}
Let $p_i\in\mathcal{P}_i$ be an arbitrary known position, with $p_i = p_i(t)$. By evaluating \eqref{eq:actual_t} at $p_i$, we obtain the actual time $\hat{t}_i(p_i)$ that CAV $i\in\mathcal{N}(t)$ is at position $p_i$. Evaluating \eqref{eq:actual_p} at the actual time, we obtain  

\begin{equation}\label{eq:actual pos Eval hat t}
     \hat{p}_{i}(\hat{t}_i(p_i)) = p_i(\hat{t}_i(p_i)) + f_i(\hat{t}_i(p_i)).     
\end{equation}
By definition of inverse function $(p_i \circ p_i^{-1})(x) = x$, and thus LHS in \eqref{eq:actual pos Eval hat t} equals to $p_i$. Substituting \eqref{eq:actual_t} in the first term of RHS, we have 
\begin{multline}\label{eq:pos Eval hat t}
     p_i = \phi_{i,3} \cdot \left(t_i(p_i)+ e_i(p_i)\right)^3 +\phi_{i,2} \cdot\left(t_i(p_i)+ e_i(p_i)\right)^2\\ + \phi_{i,1} \cdot (t_i(p_i)+ e_i(p_i)) +\phi_{i,0} + f_i(\hat{t}_i(p_i)).\quad    
\end{multline}
Next, by expanding \eqref{eq:pos Eval hat t}, we get 
\begin{align}\label{eq:pos Eval hat t Step2}
     f_i(\hat{t}_i(p_i)) &=- [\phi_{i,3}\cdot e_i(p_i)^3 + 3\phi_{i,3}\cdot e_i(p_i)^2\cdot t_i(p_i) \nonumber\\ &+\phi_{i,2}\cdot e_i(p_i)^2 +3\phi_{i,3}\cdot e_i(p_i)\cdot t_i(p_i)^2 \nonumber\\
     &+2\phi_{i,2}\cdot e_i(p_i)\cdot t_i(p_i)+\phi_{i,1} \cdot e_i(p_i)].
\end{align}
Since $p_i\in\mathcal{P}_i$ is an arbitrary known position, the above equation holds for any $p_i$, and the proof is complete.
\end{proof}

\eat{\begin{definition}\label{def:actVelTraj}
The actual speed trajectory of CAV $i\in\mathcal{N}(t)$ is a random process denoted by $\hat{v}_i(t)$ defined on probability space $(\Omega_i,\mathscr{F}_i,\mathbb{P})$.
\end{definition}}

\begin{corollary}
The actual speed trajectory of CAV $i\in\mathcal{N}(t)$,  $\hat{v}_i(t)$ is a random process defined on probability space $(\Omega_i,\mathscr{F}_i,\mathbb{P})$, and it is found from
\begin{equation}
   \hat{v}_{i}(t) = v_i(t)+ g_i(t), \label{eq:actual_v}\\
\end{equation}
where $g_i(t)$ is unknown deviation from nominal speed trajectory, and is defined on  $(\Omega_i,\mathscr{F}_i,\mathbb{P})$.
\end{corollary}
\begin{proof}
By taking time derivative of \eqref{eq:actual_p}, the result follows immediately.
\end{proof}
\subsection{Robust Coordination Problem}

For each CAV $i \in\mathcal{N}(t)$, we formulate a robust coordination problem in the presence of uncertainty. We seek to derive the new minimum time $t_i^{f}$ for CAV $i$ to exit the control zone. This exit time corresponds to a new unconstrained energy-optimal trajectory that satisfies all the state, control, and safety constraints for all realizations of uncertainty. 
In what follows, let $E_i(\cdot)\subset\mathcal{E}_i$,  $F_i(\cdot)\subset\mathcal{P}_i$, $G_i(\cdot)\subset\mathcal{V}_i$ denote the bounded confidence intervals of CAV $i$ for random process $e_i(\cdot)$, $f_i(\cdot)$, and $g_i(\cdot)$, respectively.

We enhance \eqref{eq:lateralSafety} as follows  
\begin{align}\label{eq:ChanceLateralSafety}
   &\big|\hat{t}_{i}(p_i^n) - \hat{t}_{j}(p_j^n)\big| \geq t_h, \\ 
   &\forall e_i(p_i^n)\in E_i(p_i^n), \forall e_j(p_j^n)\in E_j(p_j^n), \notag
\end{align}
to include the CAVs' deviations from their nominal time trajectories. 
This constraint should be satisfied for every possible realizations of deviation from the nominal time trajectory of CAV  $i\in\mathcal{N}(t)$ and CAV $j\in\mathcal{N}(t) \setminus \{i\}$. Note that for CAVs $i$ and $j$, $p_i^n$, $p_j^n$ are constant and they depend only on the intersection's geometry (Fig. \ref{fig:intersection}).

Similarly, we enhance rear-end safety constraint \eqref{eq:rearend} defined on the nominal trajectories by incorporating the deviations from nominal position trajectories \eqref{eq:actual_p} as follows
\begin{align} \label{eq:ChanceRearendSafety}
  &\hat{p}_k(t)-\hat{p}_i(t) \geq \hat{\delta}_i(t)=\gamma + \varphi\cdot \hat{v}_i(t),\\ 
   &\forall f_i(t)\in F_i(t),~ \forall f_k(t)\in F_k(t),~\forall g_i(t)\in G_i(t), \notag
\end{align}
where the distance between CAV $i\in\mathcal{N}(t)$ and the preceding CAV $k\in\mathcal{N}(t)\setminus\{i\}$ has to be greater than a safe distance $\hat{\delta}_i(t)$ for every realizations of deviations from the nominal trajectories of CAV $i$ and CAV $k$ (Fig. \ref{fig:intersection}). Finally, to account for the deviation of the speed of CAV $i$, we enhance the constraint 
 \begin{align}\label{eq:ChanceSpeedLimit}
        &v_{\min}\leq \hat{v}_i(t)\leq v_{\max},\\
        &\forall g_i(t) \in G_i(t)\notag, 
 \end{align}
 for every realizations of deviation from the nominal speed trajectory of CAV $i$.

\begin{problem} \label{prb:mintfProblemProbab} For each CAV $i\in\mathcal{N}(t)$, we consider the following robust coordination problem
\end{problem}
\begin{align}\label{eq:decentral}
    &\min_{t_i^f\in \mathcal{T}_i(t_i^0)} t_i^f \\
    \text{subject to: }& \eqref{eq:nominal_p},\eqref{eq:nominal_v}, \eqref{eq:nominal_u}, \eqref{eq:nominal_time}, \eqref{eq:actual_t}, \eqref{eq:actual_p}, \eqref{eq:actual_v},  \eqref{eq:ChanceLateralSafety},\eqref{eq:ChanceRearendSafety},\eqref{eq:ChanceSpeedLimit},\notag\\
    \text{and given} &\text{ boundary conditions.}\nonumber
\end{align}

\begin{remark}\label{rmk:notation}
In what follows, to simplify notation, for CAV $i\in\mathcal{N}(t)$, we denote the original nominal trajectories (resulting from the solution of Problem \ref{prb:mintfProblem}) with bar, e.g., $\Bar{t}_i(p_i),\Bar{p}_i(t),\Bar{v}_i(t)$, and $\Bar{u}_i(t)$, and reserve $t_i(p_i),p_i(t),v_i(t)$, and $u_i(t)$ for the new nominal trajectories (resulting from the solution of Problem \ref{prb:mintfProblemProbab}). 
\end{remark}

\section{Solution approach}

In our framework, upon entering the control zone, CAV $i\in\mathcal{N}(t)$ does not have any information about its uncertainty, and thus we have $E_i(\cdot)=F_i(\cdot)=G_i(\cdot) = \emptyset$. First, CAV $i$ communicates with the coordinator and obtain trajectories and information about the uncertainty of CAVs which are already in the control zone. Using this information, CAV $i$ computes its nominal trajectories by solving Problem \ref{prb:mintfProblemProbab}. 
As CAV $i$ travels following these nominal trajectories, it makes measurements (possibly noisy) of the actual time that it reaches to different positions $p_i\in\mathcal{P}_i$, denoted by $\Tilde{t}_i(p_i)\in\mathbb{R}_{\geq 0}$, and given by $\Tilde{t}_i(p_i) = \hat{t}_{i}(p_i) + \xi _i $, where $\xi _i \sim \mathcal{N}(0,\,\sigma_n^{2})$ is a Gaussian noise with unknown variance $\sigma_n^{2}$.  Next, we define the error in the time trajectory based on these measurements.     

 \begin{definition}
  The difference between the noisy measurements of the time trajectory $\Tilde{t}_i(p_i)$ and the nominal time trajectory $t_i(p_i)$ is denoted by $\Tilde{e}_i(p_i)= \Tilde{t}_i(p_i)-t_i(p_i) $, which is a random process on a probability space $(\mathcal{P}_i,\mathscr{P}_i,\mathbb{P})$. 
 \end{definition}
  \begin{definition}
 The set of observation samples of CAV $i\in\mathcal{N}(t)$, is denoted by $\mathcal{O}_i$, i.e., 
 \begin{align}
     \mathcal{O}_i = (\mathbf{p}_i,\Tilde{\mathbf{e}}_i) = \left\{\left(p_i^{(j)} , \Tilde{e}_i(p_i^{(j)})\right)\mid j=1,\dots,n\right\},
 \end{align}
 where $\mathbf{p}_i$, $\Tilde{\mathbf{e}}_i$ represent the $n$ dimensional vector of positions and corresponding observed errors in time trajectory, respectively, and the $j^{\text{th}}$ sample is denoted by $\left(p_i^{(j)} , \Tilde{e}_i(p_i^{(j)})\right)$. 
 \end{definition}

To this end, we consider that CAV $i\in\mathcal{N}(t)$ makes the noisy observations $\mathcal{O}_i$ before reaching a certain position in the control zone called \textit{uncertainty characterization point} denoted by $p^{z}\in\mathcal{P}_i$. Let $t_i^{z}$ be the actual time that CAV $i$ reaches this point. 
After CAV $i$ reaches $p^z$, it characterizes its uncertainty in the time, position, and speed trajectories based on the observed information, $\mathcal{O}_i$. Then, CAV $i$ communicates with the coordinator to access the trajectories of other CAVs $j<i \in\mathcal{N}(t)$ which entered the control zone earlier than CAV $i$. CAV $i$ also obtains information about their deviation from their trajectories, if they have already characterized their uncertainty. After obtaining this information, CAV $i$ solves the robust coordination problem (Problem \ref{prb:mintfProblemProbab}), with the revised initial conditions $p_{i}(t_i^{z}) = p^z$ and 
 $v_{i}(t_i^z) =\Bar{v}_i(t_i^{z}) + g(t_i^{z})$  (recall that based on Remark \ref{rmk:notation}, $\Bar{v}_i(t_i^z)$ is the speed of CAV $i$ at $t_i^{z}$ computed from the original nominal trajectory). CAV $i$ communicates back these new nominal trajectories along with its characterization of uncertainty to the coordinator. Then, the coordinator broadcasts a \textit{replanning event} for all CAVs $j>i \in\mathcal{N}(t)$ which entered the control zone after CAV $i$ to re-plan their trajectory with this new information. These CAVs then sequentially resolve their optimization problem, and plan their trajectory accordingly. Algorithm \ref{Alg:robust coordination} shows a psueducode of this process.

 \begin{algorithm}
\small
 \caption{Robust Coordination Psuedocode}
 \begin{algorithmic}[1]
 \State{\textit{Replanning Event} $\gets$\texttt{False}}
\For{CAV $i \in\mathcal{N}(t)$}
\If{$t=t_i^0$}
\State{$E_i, F_i, G_i \gets \emptyset$}
\State{Solve Problem \ref{prb:mintfProblemProbab}}
\ElsIf{CAV $i$ reached at $p^z$}

\State{$E_i, F_i, G_i \gets$ Characterize uncertainty based on $\mathcal{O}_i$}
\State{Solve Problem \ref{prb:mintfProblemProbab}}
\State{\textit{Replanning Event} $\gets$\texttt{True}}
\ElsIf{\textit{Replanning Event}}
\State{Solve Problem \ref{prb:mintfProblemProbab}}
\EndIf
\EndFor
\end{algorithmic} \label{Alg:robust coordination}
\end{algorithm}

 Next, we present the process for characterizing uncertainty for CAV $i\in\mathcal{N}(t)$ based on the set of observation samples, $\mathcal{O}_i$. 
 
\subsection{Modeling Uncertainty in Time Trajectory as GP}
 Let CAV $i\in\mathcal{N}(t)$ make noisy observations of actual time trajectory at different positions $p_i \in \mathcal{P}_i$ before reaching the uncertainty characterization point $p^{z}$. These observations can be obtained by the CAV or an infrastructure at certain positions in the control zone.

 Given the observation samples, $\mathcal{O}_i$, we use GP regression to model $e_i(p_i) \sim\mathcal{GP}(m(p_i),k(p_i,p_i^\prime))$, where $m(p_i)$ and $k(p_i,p_i^\prime)$ represents the prior mean and covariance, respectively. We assume no prior knowledge on the error is available, and thus, we set the prior mean to zero, $m(p_i)=0$. For the covariance function, we adopt the\textit{ Mat\'ern $3/2$} model that is one time differentiable in the mean-square sense, and it is
 given by $k(p_i,p_i^{\prime}) ={\dfrac {\sigma_s ^{2}}{\Gamma (\frac{3}{2} )}}\left(\sqrt {3 }\frac {p_i-p_i^{\prime}}{\ell_s } \right)^{\frac{3}{2} }K_{\frac{3}{2} }\left(\sqrt {3}\frac {p_i-p_i^{\prime}}{\ell_s }\right)$, where $K_{\frac{3}{2} }$ and $\Gamma$ are modified Bessel and Gamma functions, respectively.
Let $\sigma_{s}^2$ and $\ell_s$ be the process variance, the covariance function's parameter, respectively \cite{rasmussen2003gaussian}. The hyperparameters are represented by $\theta = [\sigma_{s}^2~\sigma_{n}^2~ \ell_s]^\top$, where $\sigma_{n}^2$ is the unknown variance of observation noise. The hyperparameters can be learned by maximizing the log marginal likelihood of the observation samples, i.e., $\theta^\ast = \arg\max_{\theta} \log \mathbb{P}(\Tilde{\mathbf{e}}_i|\mathbf{p}_i,\theta)$.
Given the observation samples, $\mathcal{O}_i$, the marginalized GP posterior at any arbitrary point $p_i^*$ is a univariate normal distribution, denoted by $e_i(p_i^*)\sim \mathcal{N}(\mu_e(p_i^\ast), \sigma_e^2(p_i^*) )$ defined with the following mean and variance 
\begin{align}
    \mu_e(p_i^\ast) &= m(p_i^\ast)+ k^*\left(K+\sigma_n^2I\right)^{-1}\Tilde{\mathbf{e}}_i\label{eq:GPPosMean},\\
    \sigma_e^2(p_i^*) &= k(p_i^*,p_i^*) -k^{*^\top}\left(K+\sigma_n^2I\right)^{-1}k^{*},\label{eq:GPPosVar}
\end{align}
where $K = K(\mathbf{p}_i,\mathbf{p}_i)$ denotes the $n\times n$ matrix of the covariances evaluated at all pairs of training point $\mathbf{p}_i$, and  $k^* = K(\mathbf{p}_i,p_i^*)$ is the $n\times 1$ matrix of the covariances evaluated at $n$ training points and one arbitrary point, $p_i^*$. 

\subsection{Uncertainty in Speed and Position Trajectories}
\begin{theorem}\label{thm:deviation on speed}
Deviation from the nominal speed trajectory of CAV $i\in\mathcal{N}(t)$, $g_i(t)$, follows a Chi-square distribution with one degree of freedom, and its posterior mean and variance at actual time $\hat{t}_i(p_i)$, where CAV arrives at position $p_i$ can be derived from
\begin{align}
\mu_g(\hat{t}_i(p_i)) &= a^{\prime}_1 \,\mu_e +a^{\prime}_2 \,{\left(\mu_e^2 +\sigma_e^2 \right)}, \label{eq:mu_g}\\
    \sigma^2_g(\hat{t}_i(p_i)) &= \sigma_e^2 \,{\left({a^{\prime}_1 }^2 +4\mu_e(a^{\prime}_1 \,a^{\prime}_2 +{a^{\prime}_2 }^2 \,\mu_e) +2\,{a^{\prime}_2 }^2 \,\sigma_e^2 \right)}\label{eq:sigma_g},
\end{align}
where $a^{\prime}_1=-2\phi_{i,2}-6\phi_{i,3}\cdot t_i(p_i)$ and $a^{\prime}_2=-3\phi_{i,3}$, and $\mu_e = \mu_e(p_i)$, $\sigma_e= \sigma_e(p_i)$.
\end{theorem}
\begin{proof}
Let $p_i\in\mathcal{P}_i$ be an arbitrary known position, with $p_i=p_i(t)$ and $\hat{t}_i(p_i)$ be the actual time, where CAV arrives at position $p_i$. Following the same steps as Lemma \ref{lem:actualPositionTraj}, the deviation is given by $g_i(\hat{t}_i(p_i)) =- [(2\phi_{i,2}+6\phi_{i,3}\cdot t_i(p_i))\cdot e_i(p_i) + 3\phi_{i,3}\cdot e_i(p_i)^2]$, where $e_i(p_i)$ is a univariate normal variable, and thus $g_i(\hat{t}_i(p_i))$ follows a Chi-square distribution with one degree of freedom \cite{fleishman1978method}. To derive expectation and variance of $g_i(\hat{t}_i(p_i))$, we use moment-generating function and its properties. Since $e_i(p_i)$ is a normal random variable, its moment-generating function is given by $M_e(\tau)=\exp(\tau \mu_e + \dfrac{1}{2}\sigma_e^2\tau^2)$. The $n^\text{th}$ moment of random variable $e_i(p_i)$, denoted by $\E[e_i(p_i)^n]$, can be derived from $\dfrac{d^n}{d\tau^n} M_e(\tau)\mid_{\tau=0}$. From linearity of expectation, we have $\mu_g(\hat{t}_i(p_i))=\E[g_i(\hat{t}_i(p_i))]=a^{\prime}_1\E[e_i(p_i)]+ a^{\prime}_2 \E[e_i(p_i)^2]$, where the first and second moments of $e_i(p_i)$, are given by $\mu_e$ and $\mu_e^2+\sigma_e^2$, respectively. To find the variance, we use $\sigma^2_g(\hat{t}_i(p_i)) = \E[g_i(\hat{t}_i(p_i))^2] -\E[g_i(\hat{t}_i(p_i))]^2$, where we can employ the same procedure and derive \eqref{eq:sigma_g}.
\end{proof}

\begin{corollary}
The deviation $f_i(t)$ from the nominal position trajectory of CAV $i\in\mathcal{N}(t)$  follows a cubic normal distribution.
\end{corollary}
\begin{proof}
Let $p_i\in\mathcal{P}_i$ be an arbitrary known position, with $p_i=p_i(t)$ and $\hat{t}_i(p_i)$ be the actual time, where CAV arrives at position $p_i$. From \eqref{eq:pos Eval hat t Step2}, we have $
f_i(\hat{t}_i(p_i)) = a_3\cdot e_i(p_i)^3 + a_2\cdot e(p_i)^2+a_1 \cdot e_i(p_i)$, where $e_i(p_i)$ is a normal variable,-
\begin{align}
    a_1 &= -3\phi_{i,3} \cdot t_i(p_i)^2 -2\phi_{i,2} \cdot t_i(p_i) - \phi_{i,1},\label{Cor:a_1}\\
    a_2 &= - 3\phi_{i,3} \cdot t_i(p_i)-\phi_{i,2},\label{Cor:a_2}\end{align}
    and $a_3 =-\phi_{i,3}$. Since $a_1,a_2$, and $a_3$ are not random variables, the proof is complete.
\end{proof}
\begin{proposition}\label{prop:deviation in pos}
For CAV $i\in\mathcal{N}(t)$, posterior mean and variance of $f_i(t)$ at actual time $\hat{t}_i(p_i)$, where vehicle arrives at position $p_i$ can be derived from
\begin{align}
    \mu_f(\hat{t}_i(p_i)) &= a_1 \,\mu_e +a_3 \,{\left(\mu_e^3 +3\,\mu_e \,\sigma_e^2 \right)}+a_2 \,{\left(\mu_e^2 +\sigma_e^2 \right)}\\
    \sigma^2_f(\hat{t}_i(p_i)) &=\sigma_e^2 \,[{a_1 }^2 +4\,a_1 \,a_2 \,\mu_e +6\,a_1 \,a_3 \,\mu_e^2 +  6\,a_1 \,a_3 \,\sigma_e^2 + \nonumber \\ &4\,{a_2 }^2 \,\mu_e^2 + 
     2\,{a_2 }^2 \,\sigma_e^2 +12\,a_2 \,a_3 \,\mu_e^3 +24\,a_2 \,a_3 \,\mu_e \,\sigma_e^2 \nonumber\\+&9\,{a_3 }^2 \,\mu_e^4 +36\,{a_3 }^2 \,\mu_e^2 \,\sigma_e^2 +15\,{a_3 }^2 \,\sigma_e^4 ],
\end{align}
where $\mu_e = \mu_e(p_i)$, $\sigma_e= \sigma_e(p_i)$.
\end{proposition}
\begin{proof}
The proof is similar to Theorem \ref{thm:deviation on speed}, hence it is omitted.
\end{proof}
\subsection{Confidence Bounds on trajectories}
After characterizing the uncertainty in the actual time trajectory based on the noisy observations of the actual time trajectory and leveraging GP regression, we construct a bounded confidence interval for random process $e_i(p_i)$, denoted by $E_i(p_i)\subset\mathcal{E}_i$, within which $e_i(p_i)$ lies with probability $P_e$ as follows 
\begin{align}
    E_i(p_i) &= \left[\mu_e(p_i)-z\sigma_e(p_i),~ \mu_e(p_i)+z\sigma_e(p_i)\right],\\
    z &= \sqrt{2}\erf^{-1}(P_e),
\end{align}
where $\mu_e(p_i)$ and $\sigma_e(p_i)$ are posterior mean and standard deviation of $e_i(p_i)$ at position $p_i\in\mathcal{P}_i$, respectively, and $\erf^{-1}(\cdot)$ is the inverse error function. 
Using Chebyshev's inequality, we construct a bounded confidence interval for the random process $f_i(t)$, denoted by $F_i(t)\subset\mathcal{P}_i$, within which $f_i(t)$ lies with at least probability $P_f$.
\begin{align}
    &F_i(p_i) = \left[\mu_f(t)-z\sigma_f(t),~ \mu_f(t)+z\sigma_f(t)\right],\\
    &\mathbb{P}\left(f_i(t)\in F_i(t)\right)\geq P_f= 1-\dfrac{1}{z^2}.
\end{align}
Deviation from speed trajectory, $g_i(t)$, follows a Chi-squared distribution which is a uni-modal distribution, i.e., its distribution permits a Lebesgue density that is non-decreasing up  to a mode and non-increasing thereafter. This unimodality allows us to employ a tighter bound for the confidence interval using Vysochanskii-Petunin inequality \cite{pukelsheim1994three}. We construct a bounded confidence interval for the random process $g_i(t)$, denoted by $G_i(t)\subset\mathcal{V}_i$, within which $g_i(t)$ lies with at least probability $P_g$ as follows
\begin{align}
    &G_i(p_i) = \left[\mu_g(t)-z\sigma_g(t),~ \mu_g(t)+z\sigma_g(t)\right],\\
    &\mathbb{P}\left(g_i(t)\in G_i(t)\right)\geq P_g= 1-\dfrac{4}{9z^2}.
\end{align}
The solution to Problem \ref{prb:mintfProblemProbab} is the optimal nominal trajectories for CAV $i\in\mathcal{N}(t)$ satisfying the safety constraints in the presence of uncertainty, which is modeled through GP based on the possibly noisy observations of the time trajectory. 
\eat{
\begin{remark}
Deterministic trajectories  $\Bar{t}_i(p_i),\Bar{p}_i(t),\Bar{v}_i(t)$, and $\Bar{u}_i(t)$, resulted from Problem \ref{prb:mintfProblem}, that do not violate any constraints in the presence of uncertainty, would be equal to the nominal trajectories obtained from Problem $\ref{prb:mintfProblemProbab}$, $t_i(p_i),p_i(t),v_i(t)$, and $u_i(t)$.
\end{remark}}

\section{Simulation Results}

To demonstrate the effectiveness of our proposed framework, we investigate the coordination of $24$ CAVs at a signal-free intersection shown in fig.\ref{fig:intersection}. The CAVs enter the control zone from $6$ different paths shown in fig. \ref{fig:intersection} with the total rate of $3600$ veh/hour and their initial speeds uniformly distributed between $12$ m/s to $14$ m/s. 
We use the following parameters for the simulation: $t_h = 0.5 $ s, $v_{\min}=0.25$ m/s, $v_{\max}=30$ m/s, $u_{\max}=2$ m/s$^2$, $u_{\min}=-2$ m/s$^2$, $\gamma=1.5$ m, $\varphi=0.5$ s, $p^z=50$ m, $P_e=P_f=P_g= 95$\%. Videos from our simulation analysis can be found at the supplemental site, \url{https://sites.google.com/view/ud-ids-lab/RBST}.

We consider the actual deviation from the nominal time trajectory for CAV $i\in\mathcal{N}(t)$ is given by function $e_i(p_i) = 0.012\log\left(1+p_i\right)^{1.5}$ which is not known to CAV $i$ a priori. Each CAV by using a GP regression, makes $n=50$ observations of this deviation before reaching at uncertainty characterization point. Upon reaching the uncertainty characterization point, it obtains the posterior deviations in time trajectory. By employing Theorem \ref{thm:deviation on speed} and Proposition \ref{prop:deviation in pos}, CAV $i$ computes first two moments of deviations in speed and position trajectories. Then, it constructs the $95$\% bounded confidence intervals which can be used to create a tube around the nominal trajectories. Figs. \ref{fig:devInTime} and \ref{fig:DevinPos} demonstrate the posterior deviations in time and position along with corresponding $95$\% confidence bounds after observing $50$ samples of actual deviation in time trajectory.

\begin{figure}[ht]
    \centering
\includegraphics[width=0.95\linewidth]{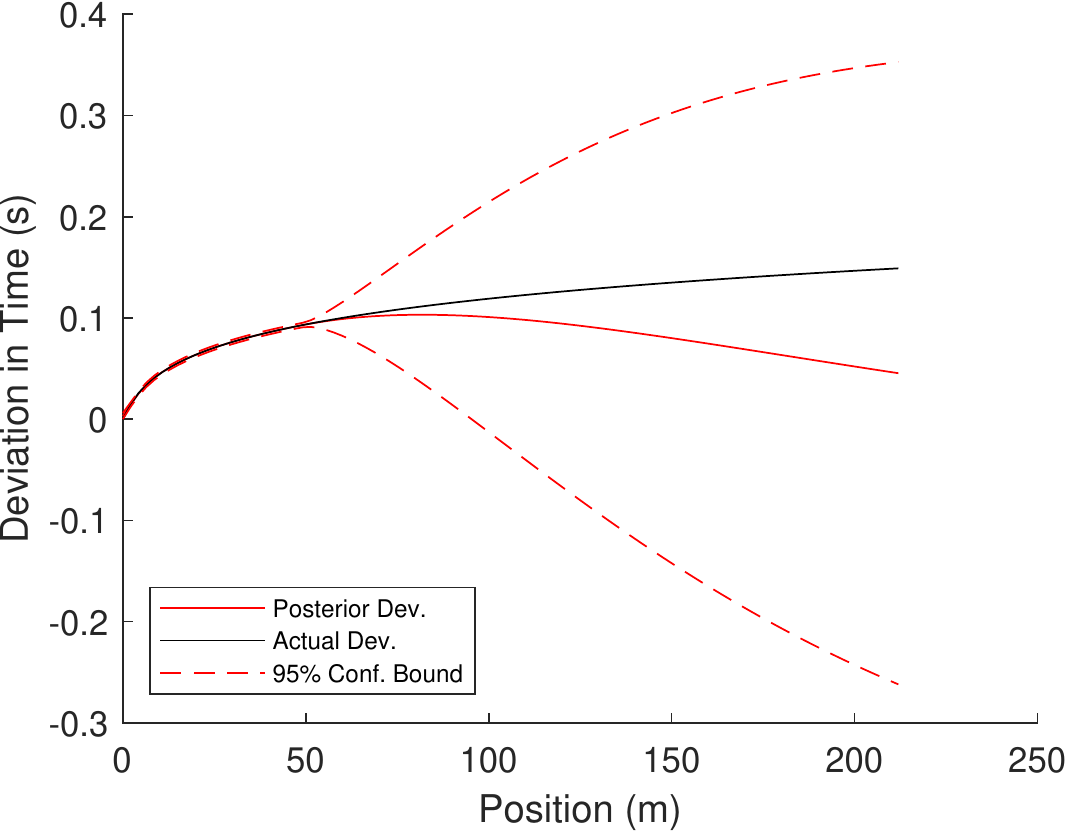}    \caption{Deviation in time trajectory.}
    \label{fig:devInTime}
\end{figure}

\begin{figure}[ht]
    \centering
\includegraphics[width=0.95\linewidth]{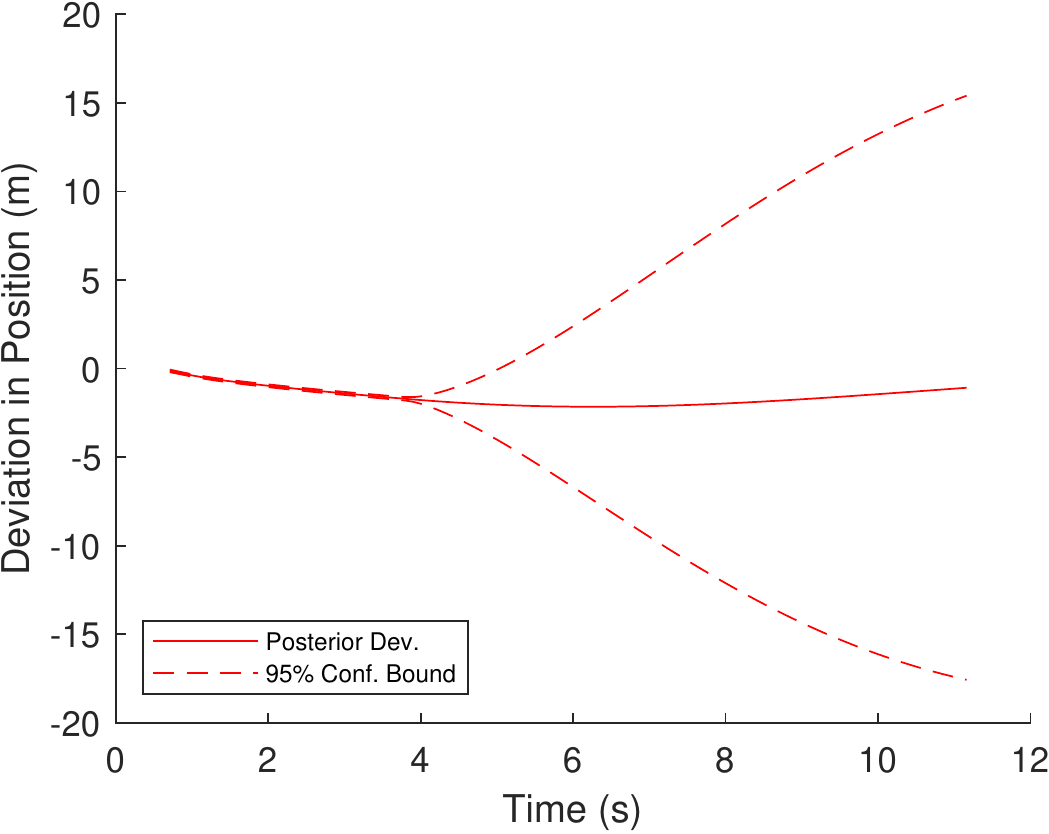}    \caption{Deviation in position trajectory.}
    \label{fig:DevinPos}
\end{figure}

Fig. \ref{fig:timeTrajDeviationProposed} illustrates time trajectories of CAVs traveling from westbound to eastbound. The CAVs nominal trajectories on this path are denoted with solid lines, while their corresponding rear-end safety constraints \eqref{eq:ChanceRearendSafety} are visualized with dotted lines in the same color. Replanning events that are due to the change of the trajectories of other CAVs in the control zone are shown with blue asterisks, and the uncertainty characterization point at which each CAV quantifies its trajectory is shown with a black square marker.  The $95$\% confidence bounds of the time trajectories are shown with dashed lines. Moreover, the CAVs from other paths that have the potential for lateral collision with CAVs in this path are shown with vertical thicker lines. Their arrival times at the conflict points with $95$\% confidence bound are shown in red, and corresponding lateral safety constraints \eqref{eq:ChanceLateralSafety} are shown with vertical black lines. This figure shows that by using our robust framework, CAVs' nominal trajectories satisfy the safety constraints for every realization of the deviations from the nominal trajectories. Fig. \ref{fig:timeTrajDeviationBaseline} visualize the case where CAVs stick to their initial planned nominal trajectories, ignoring uncertainty. It can be seen that for multiple cases, the trajectories of the CAVs with $95$\% confidence bounds cross the vertical lines representing the lateral constraints.

\begin{figure}[ht]
    \centering
\includegraphics[width=0.99\linewidth]{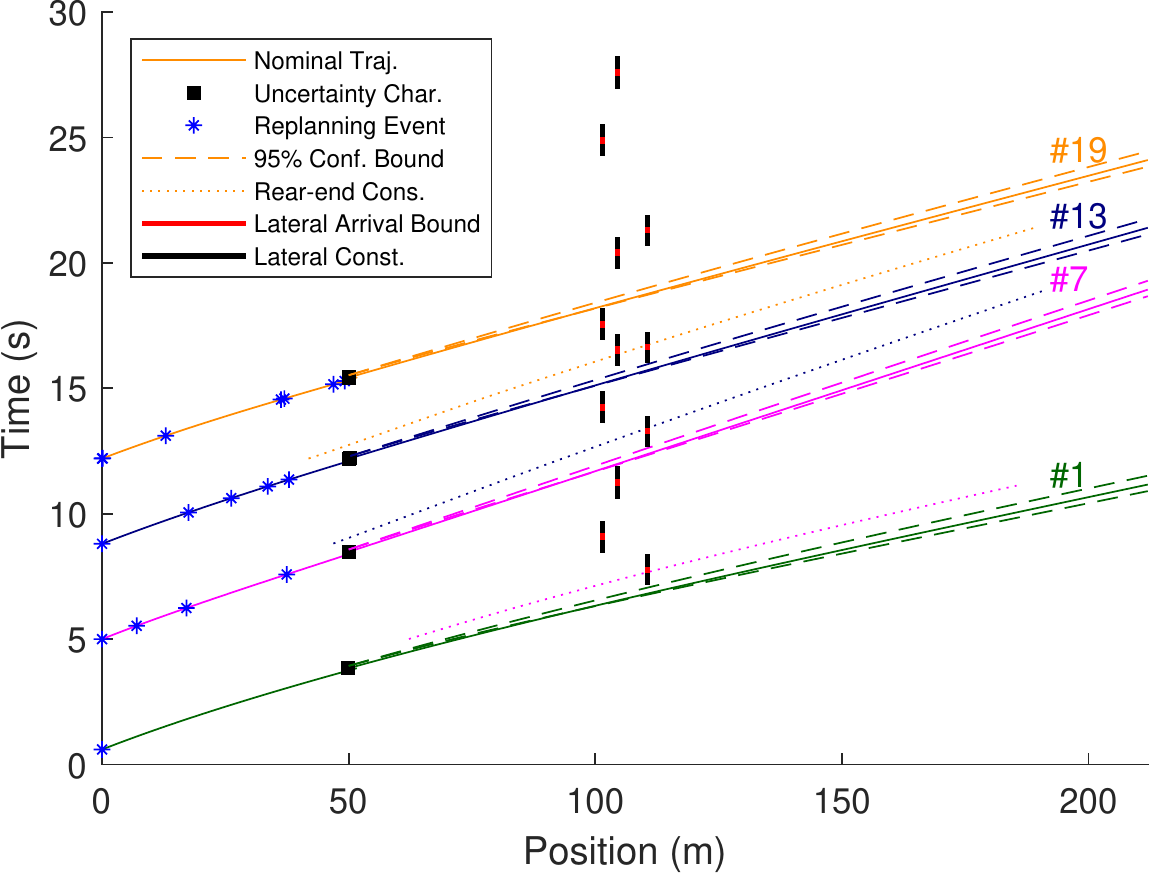}    \caption{Time trajectories of CAVs under robust coordination.}
    \label{fig:timeTrajDeviationProposed}
\end{figure}

\begin{figure}[ht]
    \centering
\includegraphics[width=0.99\linewidth]{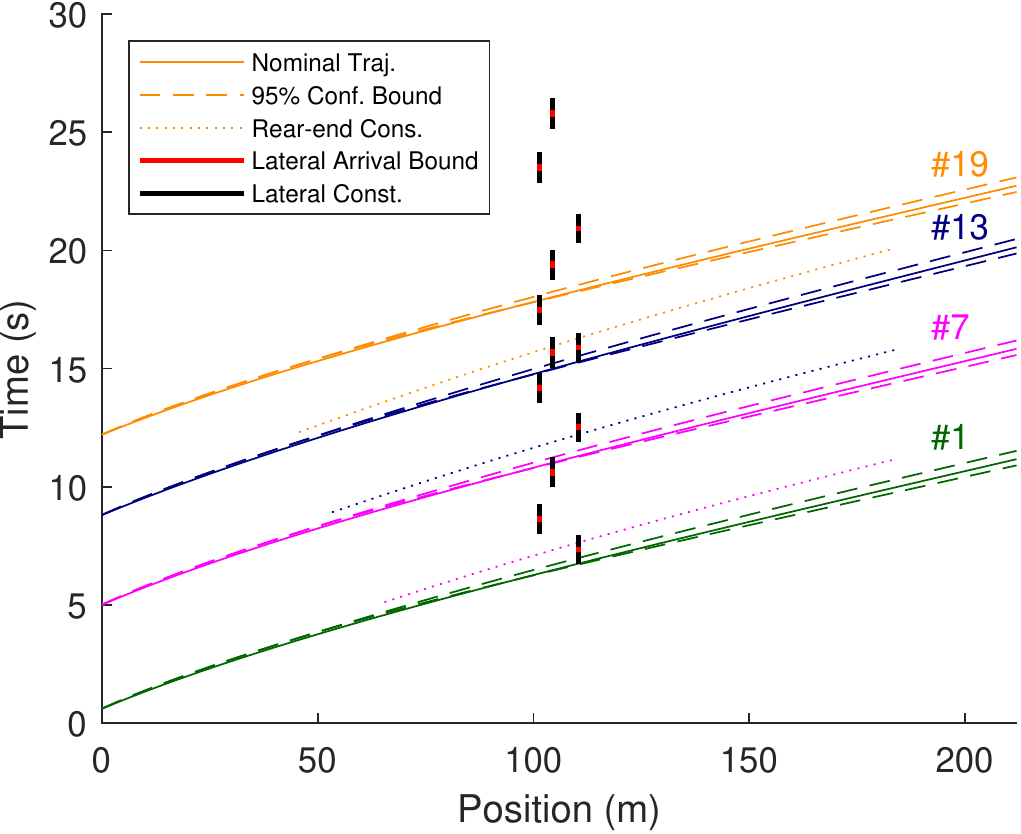}    \caption{Time trajectories of CAVs under deterministic coordination.}
    \label{fig:timeTrajDeviationBaseline}
\end{figure}

We demonstrate the control input trajectory of CAV \#$19$ traveling from westbound to eastbound under the robust coordination framework in Fig. \ref{fig:ControlInputTraj}. It can be seen that at replanning events shown with blue asterisks, CAV \#$19$ solves Problem \ref{prb:mintfProblemProbab} with updated trajectories of other CAVs and new information about their uncertainty. Upon reaching the uncertainty characterization point shown with a black square, CAV \#$19$ learns the deviation in its nominal trajectories and solves Problem \ref{prb:mintfProblemProbab} with this new information. It communicates back the new nominal trajectories along with its characterization of uncertainty to the coordinator. Then, the coordinator broadcasts a replanning event for all CAVs which entered the control zone after CAV \#$19$ to replan their trajectories.

\begin{figure}[ht]
    \centering
\includegraphics[width=0.99\linewidth]{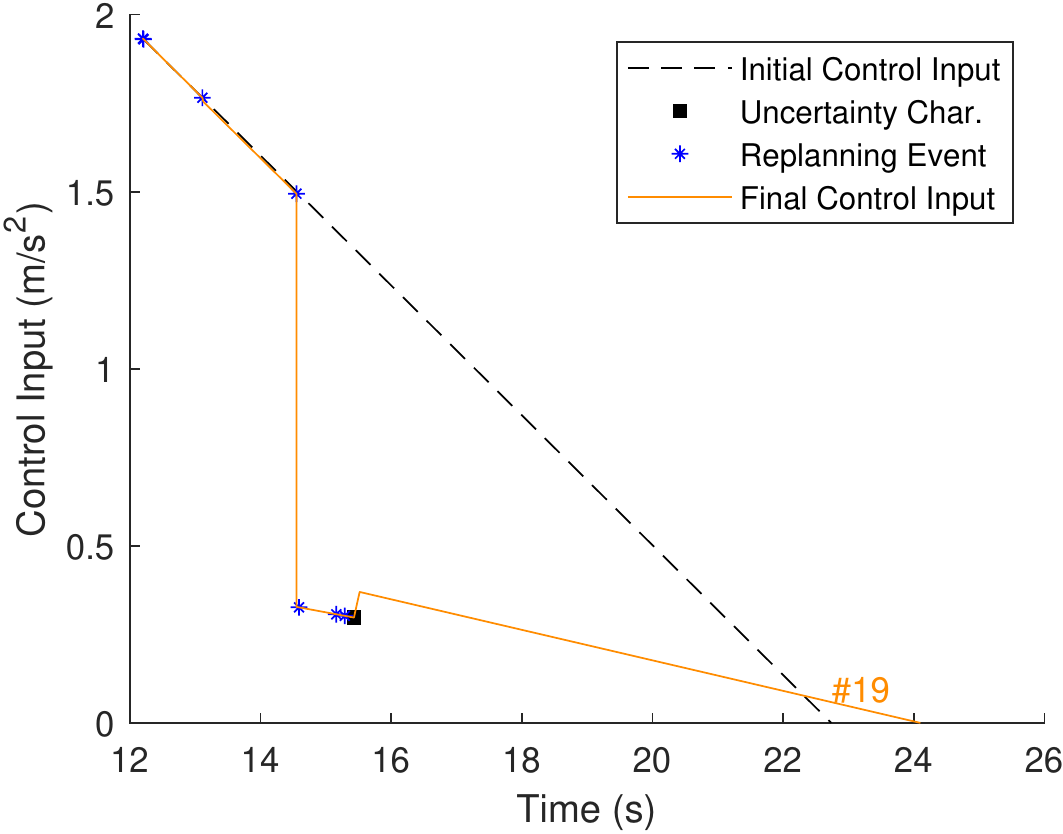}    \caption{Control input trajectory of a CAV under robust coordination. }
    \label{fig:ControlInputTraj}
\end{figure}

\section{Concluding Remarks and Discussion} 
In this paper, we extended a framework we developed earlier for coordination of CAVs in \cite{Malikopoulos2020} to include the deviations from the nominal trajectories as uncertainty and reformulated it as a robust coordination problem. We adopted the data-driven approach, GP regression, to learn the uncertainty from the possibly noisy observation of CAVs' time trajectories. After obtaining the statistical knowledge about the deviation from nominal trajectories, we constructed the confidence interval for time, position, and speed trajectories using the inverse error function, Chebyshev's inequality, and  Vysochanskii-Petunin inequality, respectively. We, finally, demonstrated the effectiveness of our proposed framework through a numerical simulation. Some potential directions for future research include considering delay in the communication and validating the proposed framework in the experimental testbed.

\bibliographystyle{IEEEtran.bst} 
\bibliography{reference/IDS_Publications_09142021.bib, reference/ref.bib}

\end{document}